\documentclass[10pt]{amsart}
\usepackage[all]{xy}
\begin{document}

\title{Pattern avoidance and $K$-orbit closures}
\author{William M. McGovern}

\maketitle
\bigskip
\begin{abstract}
We extend the characterization of smooth and rationally smooth classical Schubert varieties by pattern avoidance to classical symmetric varieties.  We also parametrize these varieties combinatorially and show how to compute the partial order by containment of closures on them.
\end{abstract}
\section{Introduction}
\newtheorem{remark}{Remark}[section]
\newtheorem{theorem}{Theorem}[section]
\newtheorem{definition}{Definition}[section]
\newtheorem{example}{Example}[section]
\newtheorem{conjecture}{Conjecture}[section]
\newtheorem{lemma}{Lemma}[section]
\newtheorem{proposition}{Proposition}[section]
\bigskip
Let $G$ be a complex semisimple Lie group and $B$ a Borel subgroup.  The quotient space $G/B$ is called the {\bf flag variety} or {\bf flag manifold} of $G$; it is well known that this space can be identified with the set $\mathcal B$ of Borel subalgebras of the Lie algebra $\mathfrak g$ of $G$, any two such subalgebras being conjugate under $G$.  The set $B\backslash G/B$ of $B$-orbits in $G/B$ is well known to be finite; it is parametrized by the Weyl group $W$ of $G$.  The $B$-orbits $C_w=BwB$ in the flag variety are called {\bf Schubert cells} and their  
closures $X_w=\overline{C_w}$ (in either the Zariski or the Euclidean topology) are called {\bf Schubert varieties}.  Schubert varieties can have very complicated singularities but by now these singularities are fairly well understood (see e.g. \cite{BL00,AB14}).  In particular, if $G$ is of classical type then it is well known that the elements of $W$ may be regarded as permutations or signed permutations.  Both of these have of course been extensively studied by combinatorists and so one could hope for a combinatorial characterization of the $w\in W$ for which $X_w$ is singular, or rationally singular (that is, has nonvanishing relative cohomology in more than one degree at some point).  That hope has been abundantly fulfilled; in this context the notions of pattern inclusion and avoidance have proven quite useful.  Writing a typical permutation $w$ in the symmetric group $S_n$ on $n$ letters in one-line notation as $[w_1\ldots w_n]$, where $w_i=w(i)$.  Recall that this permutation is said to {\bf include} the pattern $v=[b_1\ldots b_m]$ (another permutation in one-line notation) if there are indices $1\le i_1<\ldots<i_m\le n$ such that $w({i_j})<w({i_k})$ if and only if $b_j<b_k$.   We say that $w$ {\bf avoids} $v$ if it does not include the latter.  If instead $w,v$ are signed permutations (lying in the hyperoctahedral group $H_n$ on $n$ letters), then the condition for $w$ to include $v$ is slightly more complicated; see \cite[Def.\ 8.3.15]{BL00}.  We will see in this article that somewhat different notions of inclusion and avoidance arise for the varieties discussed here.  One typically characterizes the permutations or signed permutations corresponding to smooth, or to rationally smooth, Schubert varieties as those avoiding all patterns in a certain list; this was first done for Schubert varieties in type $A$ in \cite{LS90}, improving on an earlier and independent classification in \cite{W89} that did not use pattern avoidance.  In this case there are just two bad patterns, namely $3412$ and $4231$.  In the other classical types there is a longer list of bad patterns corresponding to singular Schubert varieties (see Chapter 13 of \cite{BL00}), but in all cases the bad patterns involve at most four indices.

A similar situation arises when one lets $K$ be a symmetric subgroup of $G$, that is, the fixed points of an automorphism $\iota$ of $G$ of order 2, and looks at the orbits of $K$ rather than $B$ on $G/B$.  Such orbits are called {\bf symmetric} and their closures are called {\bf symmetric varieties}.  It is well known that there are only finitely many symmetric orbits; they play a crucial role in the representation theory of the real form of $G$ corresponding to $K$ analogous to that played by the $B$-orbits in the theory of highest weight modules for the Lie algebra of $G$ \cite{BB81,LV83,V83,BoB85}.  This chapter is devoted to such orbits and the singularities of their closures in the classical cases; they are of intrinsic interest apart from their connection to representation theory.

In more detail, this chapter is divided into this introduction plus four other sections.  In the first of these we classify the symmetric orbits for all classical groups $G$, running through the list of symmetric subgroups $K$ up to automorphism for each such $G$ and parametrizing $K$-orbits in $G/B$ in each case by involutions in symmetric groups, sometimes with extra structure.  In the next one we work out the partial order on varieties by containment of closures in each case, sometimes by relating this order to Bruhat-Chevalley order on permutations (regarded as elements of a  Weyl group of type $A$), and sometimes by giving combinatorial \lq\lq moves" generating this order.  In all cases the poset of varieties is ranked by the shifted dimension of the variety; we give an explicit formula for this dimension in each case.  Section 4 is the heart of the chapter. In it we give pattern avoidance criteria for classical symmetric varieties to be smooth or rationally smooth, very similar in spirit to the ones given by Billey, Lakshmibai, and others for classical Schubert varieties (see \cite{BL00}).  These are proved by techniques due to Carrell, Peterson, Brion, Kumar, and others, similar to those used in the Schubert case, but a number of new ideas and complications arise.  In the last section we broaden our horizons, considering 
the properties of being a local complete intersection or Gorenstein in addition to smoothness and rational smoothness in one family of cases.  Here an unexpected connection to Richardson varieties (intersections of Schubert and opposite Schubert varieties) permits the explicit computation of Kazhdan-Lustig-Vogan polynomials, though Gorensteinness does not in general seem to be characterizable by pattern avoidance.

\section{Symmetric orbits in classical flag varieties}
\bigskip
We begin by running through the well-known list of symmetric subgroups $K$ arising for each classical simple group $G$ (up to conjugacy in $G$) and parametrizing the $K$-orbits in $G/B$ in each case. These orbits were first studied for general $G$ by Matsuki, Springer and others.  Later work specifically in the classical cases was done by Matuski-Oshima \cite{MO88} and Yamamoto \cite{Y97,Y97'}. 

\subsection{The general linear group}
Take first $G=GL(n,\mathbb C)$; it turns out to be more convenient to work with this reductive group rather than its simple subgroup $SL(n,\mathbb C)$.  For any positive integers $p,q$ with $p+q=n$ let $\iota=\iota_{p,q}$ be conjugation by a diagonal matrix with $p$ eigenvalues $1$ and $q$ eigenvalues $-1$, so that $K\cong GL(p,\mathbb C)\times GL(q,\mathbb C)$.  This case is labelled {\bf type $AIII$} in the literature and corresponds to the real form $U(p,q)$ of $G$, the group of isometries of a Hermitian form of signature $(p,q)$ on $\mathbb C^{p+q}$ (see \cite[Ch.\ X]{He78}). 

\begin{definition}
A {\bf clan} of signature $(p,q)$ is a sequence $(c_1,\ldots,c_{p+q})$ of $p+q$ symbols such that each $c_i$ is either $+$ or $-$ or a natural number, such that every natural number occurs either exactly twice or not at all among the $c_i$  and the number of distinct natural numbers and $+$ signs among the $c_i$ is $p$.  Two clans are identified whenever they have the same signs in the same positions and pairs of equal numbers in the same positions (so that for example $(1+1-)$ is identified with $(2+2-)$ but not with $(1+-1)$).
\end{definition}

Thus for example, the clans of signature $(2,1)$ are exactly $(1+1),(+11),(11+),\break (++-),(+-+)$, and $(-++)$.

\begin{theorem}[{\cite{MO88,Y97}}]
The $GL(p,\mathbb C)\times GL(q,\mathbb C)$-orbits in $GL(p+q,\mathbb C)/B$ are parametrized by clans of signature $(p,q)$.
\end{theorem}

Thus there are indeed only finitely many clans of a fixed length up to the identification defined in Definition 2.1.  The parametrization is obtained as follows.  It is well known that $G/B$ may be identified with the space of complete flags $V_0\subset V_1\cdots\subset V_n$ in $\mathbb C^n$.  Fix any basis $v_1,\ldots,v_n$ of $\mathbb C^n$ and let $P$ be the span of the first $p$ vectors and $Q$ the span of the remaining vectors in this basis.  Embed the product $GL(P)\times GL(Q)$ in an obvious way as a subgroup $K$ of $G$.  Then the $K$-orbit $\mathcal O_c$ in $G/B$ corresponding to the clan $c=(c_1,\ldots,c_n)$ consists of all flags $V_0\subset\cdots\subset V_n$ such that the dimension of $V_i\cap P$ equals the number of $+$ signs and pairs of equal numbers among $c_1\ldots c_i$, while the dimension of $V_i\cap Q$ equals the number of $-$ signs and pairs of equal numbers occurring among $c_1\ldots c_i$, for all indices $i$ between 1 and $n$. In both cases we do not count numbers appearing only once among the relevant $c_i$.

For general $G$, there is a unique open $K$-orbit in any flag variety $G/B$, having dimension equal to that of $G/B$, since there are only finitely many $K$-orbits in $G/B$.  Whenever $K$-orbits are parametrized by clans, we denote the clan corresponding to the open orbit by $c_o$.

\begin{example}  The clan $(1+-1-)$ corresponds to the set of all flags $V_0,\ldots,V_5$ such that the dimensions of $V_i\cap P,V_i\cap Q$ are $0,0,1,0,1,1,2,2,2,3$, respectively, for $i=1,2,3,4,5$.  Thus the first two dimensions are 0 since there are no signs are pairs of equal numbers in the first entry of the clan, the next dimension is 1 since there is one $+$ sign in the first two entries of the clan and no pairs of equal numbers, and so on.  In general, if $p\ge q$, then the clan corresponding to the open orbit is $(1\ldots q+\ldots+q\ldots1)$, with $p-q$ plus signs (we will see in Theorem 3.5 below that the closure of this orbit contains that of all the others, so is the full flag variety). 
\end{example}

Next we consider the case where $\iota$ sends a matrix $M$ to its inverse transpose.  Then $K=O(n,\mathbb C)$, the isometry group of the usual dot product, a symmetric bilinear form on $\mathbb C^n$.  This case corresponds to the real form $GL(n,\mathbb R)$ of $G$ and has the label {\bf type $AI$} in the literature.  

\begin{theorem}[{\cite[Example 10.2]{RS90}}]
For $G=GL(n,\mathbb C), O(n,\mathbb C)$-orbits on $G/B$ are parametrized by the set $I_n$ of involutions in the symmetric group $S_n$ on $n$ letters.  The $O(n,\mathbb C$-orbit $\mathcal O_\pi$ corresponding to an involution $\pi=[\pi_1\ldots\pi_n]$ consists of all flags $V_0\subset\cdots\subset V_n$ such that the rank of the dot product on the Cartesian product $V_j\times V_k$ equals the number of pairs $\ell\le j$ with $\pi_\ell\le k$, for all indices $j,k$ between 1 and $n$. 
\end{theorem}

\begin{example}
The involution corresponding to the open orbit has one-line notation $[1\ldots n]$.  We will see in Theorem 3.6 below that the partial order on orbits given by containment of their closures corresponds to the reverse Bruhat order on involutions, so that the identity involution, lying at the bottom of the poset of involutions with the usual Bruhat order, lies at the top of this poset with reverse Bruhat order.
\end{example}

The last type (labelled $AII$) arising for $GL(n,\mathbb C)$ corresponds to the involution $\iota$ sending a $2m\times 2m$ matrix $M$ to $J^{-1}(M^{-1})^tJ=-J(M^{-1})^tJ$, where $J$ is the block matrix $\begin{pmatrix} 0&I\\ -I&0\end{pmatrix}$, where $I$ is the $m\times m$ identity matrix.  The real form of $GL(2m,\mathbb C)$ corresponding to this type is $U^*(2m)$ \cite[Ch.\ X]{He78}, which may be identified with the group of $m\times m$ invertible matrices over the quaternions $\mathbb H$.  Then $K\cong Sp(2m,\mathbb C)$, the isometry group of nondegenerate skew-symmetric bilinear form on $\mathbb C^{2m}$.  

\begin{theorem}[{\cite{CCT15}}]
$Sp(2m,\mathbb C)$-orbits on $GL(2m,\mathbb C)/B$ are parametrized by the set $I_{2m}'$ of involutions  $\pi$ in $S_{2m}$ without fixed points.  
\end{theorem}

\begin{example}
The involution corresponding to the open orbit is\hfil\break $(2143\ldots 2m,2m-1)$; again this is the unique largest fixed-point-free involution in the reverse Bruhat order (see Theorem 3.7 below).
\end{example}

\subsection{Symplectic groups}

Turning now to the group $G=Sp(2m,\mathbb C)$ of type $C$, we attach to any pair $(p,q)$ of positive integers with $p+q=m$ the involution $\iota_{p,q}'$ given by conjugation by a diagonal matrix in $G$ with $2p$ eigenvalues 1 and $2q$ eigenvalues $-1$; here the type label is $CII$.  The corresponding real form of $G$ is $Sp(p,q)$ \cite[Ch.\ X]{He78}; it consists of all matrices in $G$ preserving a suitable Hermitian form of signature $(2p,2q)$.  Here $K\cong Sp(2p,\mathbb C)\times Sp(2q,\mathbb C)$.  Here we are fixing two subspaces $P,Q$ of $\mathbb C^{2m}$, of respective dimensions $2p,2q$, such that the restrictions of the symplectic form to $P$ and $Q$ are both nondegenerate and $P$ and $Q$ are orthogonal under this form.  

\begin{definition}
A {\bf symmetric clan} of signature $(2p,2q)$ is a clan $(c_1,\ldots,c_{2p+2q})$ of signature $(2p,2q)$ such that if $c_i$ is a sign for some $i$, then $c_{2p+2q+1-i}$ is the same sign, while if $c_i, c_j$ is a pair of equal numbers then $i+j\ne 2p+2q+1$ and $c_{2p+2q+1-i},c_{2p+2q+1-j}$ is another such pair.
\end{definition}

\noindent Note that the reverse $c^r=(c_{2p+2q},\ldots,c_1)$ of a symmetric clan $c=(c_1,\ldots,c_{2p+2q})$ of signature $(2p,2q)$ is another such clan.

\begin{theorem}[{\cite{MO88,Y97}}]
$Sp(2p,\mathbb C)\times Sp(2q,\mathbb C)$-orbits in $Sp(2p+2q,\mathbb C)/B$ are parametrized by symmetric clans of signature $(2p,2q)$. 
\end{theorem}

Recall that the flag variety of $G$ may be identified with the set of isotropic flags $V_0\subset\cdots\subset V_m$ in $\mathbb C^{2m}$, so that the dimension $v_i$ is $i$ and each $V_i$ is isotropic with respect to the form.  Any such flag extends canonically to a complete flag $V_0\subset\ldots\subset V_{2m}$ of $\mathbb C^{2m}$ if we let $V_{2m-j}$ be the orthogonal $V_j^\perp$ of $V_j$ under the form for $1\le j\le m$.  Then the flags lying in a specified clan $(c_1\ldots c_{2p+2q})$ are those whose intersections $V_i\cap P,V_i\cap Q$ satisfy the same conditions as in the paragraph before Example 2.1.

\begin{example}
There are just four symmetric clans of signature $(2,2)$, namely $(1212),(1122),(+--)$, and $(-++-)$.  In general, the clan corresponding to the open orbit is the same as in type $AIII$.
\end{example} 

The other possibility for $\iota$ with this group $G$ is conjugation by the block matrix $\begin{pmatrix}I&0\\ 0&-I\end{pmatrix}$, where the skew-symmetric form $(\cdot,\cdot)$ is taken to correspond to the $2m\times 2m$ matrix $\begin{pmatrix} 0&I\\-I&0\end{pmatrix}$. The corresponding real form of $G$ is $Sp(2m,\mathbb R)$ and the type label is $CI$.  Here $K\cong GL(m,\mathbb C)$.

\begin{definition}
A {\bf skew-symmetric clan of length $2m$} is a clan $(c_1,\ldots,c_{2m})$ of signature $(m,m)$ such that if $c_i$ is a sign, then $c_{2m+1-i}$ is the opposite sign and if $c_i,c_j$ is a pair of equal numbers, then $c_{2m+1-i},c_{2m+1-j}$ is another such pair.  Here there is no requirement that $i+j\ne 2m+1$ if $c_i,c_j$ are a pair of equal numbers.
\end{definition}

\begin{theorem}[{\cite{MO88,Y97,Y97'}}]
$GL(m,\mathbb C)$-orbits in $Sp(2m,\mathbb C)/B$ are parametrized by skew-symmetric clans $c=(c_1\ldots c_{2m})$ of length $2m$.
\end{theorem}

Here we fix a maximal isotropic subspace $P$  of $\mathbb C^{2m}$, of dimension $m$ and an isotropic dual space $Q$ to $P$, so that $\dim Q=m$ and $P,Q$ are paired nondegenerately by the form.  Given an isotropic flag $V_0\subset\cdots\subset V_m$ in $\mathbb C^{2m}$, extended as above to a complete flag $V_0\subset\cdots\subset V_{2m}$ in $\mathbb C^{2m}$, the condition on the intersections $V_i\cap P,V_i\cap Q$ for this flag to lie in the orbit with a given skew-symmetric clan are the same as in the previous case (even though $P$ are $Q$ are isotropic here whereas before they were nondegenerate under the form).

\begin{example}
There are eleven skew-symmetric clans of length 4.  Of these, four involve only signs, namely $(++--),(+-+-),(-+-+)$, and $(--++)$.  Three involve only numbers:  $(1122),(1212),(1221)$.  The remaining four, $(1+-1),\break (1-+1),(+11-)$, and $(-11+)$, involve both signs and numbers.  In general, the clan $c_o$ corresponding to the open orbit is $(1,\ldots,m,m,\ldots,1)$; we will see in the next section that it is the only one whose dimension matches that of the flag variety.
\end{example}

\subsection{Orthogonal groups}

Finally we consider orthogonal groups $G=O(n,\mathbb C)$ (which are more convenient to work with than their simple counterparts $SO(n,\mathbb C)$).  As usual let $(\cdot,\cdot)$ be the ambient form (this time symmetric) of which $G$ is the isometry group.  As for $GL(n,\mathbb C)$ the first possibility (with type label $BDI$) for $\iota$ is $\iota_{p,q}\in G$, defined as in type $AIII$ above.  The corresponding real form of $G$ is $O(p,q)$, the isometry group of a symmetric form of signature $(p,q)$ on $\mathbb R^{p+q}$.  Here we have $K=O(p,\mathbb C)\times O(q,\mathbb C)$.

\begin{definition}
An {\bf orthosymmetric} clan of signature $(p,q)$ is a symmetric clan $c=(c_1,\ldots,c_{p+q})$ of signature $(p,q)$, except that we allow $c_s=c_{p+q+1-s}\in\mathbb N$ for an index $s$, while if $n=p+q=2m+1$ is odd we require that $c_{m+1}$ be a sign.
\end{definition}

\begin{theorem}[{\cite{MO88,Y97'}}]
$O(p,\mathbb C)\times O(q,\mathbb C)$-orbits on $O(p+q,\mathbb C)/B$ are parametrized by orthosymmetric clans of signature $(p,q)$. 
\end{theorem}

As in type $AIII$ we fix subspaces $P,Q$, of respective dimensions $p,q$, that are nondegenerate and orthogonal under the form (which is now symmetric).  An element of the flag variety corresponds to a maximal flag $V_0\subset\cdots\subset V_{\lfloor (p+q)/2\rfloor}$ of isotropic subspaces under the form, extended as above to a complete flag in $\mathbb C^{p+q}$.  The criterion for a flag to lie in the orbit with a given clan is the same as for type $CII$ above.  The open orbit corresponds to the same clan as for type $AIII$ above.

\begin{remark}
Orbits corresponding to clans with at least one sign, not in the middle position $m+1$ if $p+q=2m+1$ is odd, split into two suborbits under the $S(O(p,\mathbb C)\times O(q,\mathbb C))$ action.
\end{remark}

In the last case, with type label $DIII$, we have $G=O(2n,\mathbb C)$ and we take the symmetric form $(\cdot,\cdot)$ to correspond to the matrix $\begin{pmatrix} 0&I\\I&0\end{pmatrix}$. Then $\iota$ is conjugation by $\begin{pmatrix} I&0\\0&-I\end{pmatrix}$.  The corresponding real form is $O^*(2n)$, the group of matrices preserving both $(\cdot,\cdot)$ and a suitable skew-Hermitian form.  Here $K\cong GL(n,\mathbb C)$.

\begin{definition}
An {\bf even skew-symmetric clan of length $2n$} is a skew-symmetric clan $c=(c_1,\ldots,c_{2n})$ such that if $c_i,c_j$ are a pair of equal numbers, then\hfil\break $i+j\ne 2n+1$ and, in addition, the number of + signs and pairs of equal numbers among the first $n$ entries $(c_1\ldots c_m)$ is even.
\end{definition}

\begin{theorem}[{\cite{MT09}}]
$GL(m,\mathbb C)$-orbits in $O(2m,\mathbb C)$ are parametrized by even skew-symmetric clans of length $2m$.
\end{theorem}

\noindent As in type $CI$, we fix a maximal isotropic subspace $P$ and take $Q$ to be isotropic dual to this subspace.  The criterion for a flag to lie in the orbit corresponding to a fixed clan is then the same as in type $CI$.

\begin{example}
Of the clans $(12++--12),(12+-+-12)$ the first is even skew-symmetric while the second is not.  The clan $c_o$ corresponding to the open orbit is $(1,\ldots,2m,2m-1,2m,\ldots,1,2)$ if $n=2m$ is even and\hfil\break $(1,\ldots,2m,-,+,2m-1,2m,\ldots,1,2)$ if $n=2m+1$ is odd.  We denote by $-c_0$ the clan obtained from $c_o$ by changing its signs if $n$ is odd; otherwise, we set $-c_0=c_o$.
\end{example}

\subsection{Summary}
\noindent We summarize these parametrizations in the following table.
\vskip .3in

\begin{tabular}{c c c}
\text{\bf type}&$K$&\text{\bf orbit parameters}\\
\hline
$AI$&$O(n,\mathbb C)$&\text{involutions}\\
$AII$&$Sp(2n,\mathbb C)$&\text{fixed-point-free involutions}\\
$AIII$&$GL(p,\mathbb C)\times GL(q,\mathbb C)$&\text{clans, signature }$(p,q)$\\
$CI$&$GL(n,\mathbb C)$&\text{skew-symmetric clans}\\
$CII$&$Sp(2p,\mathbb C)\times Sp(2q,\mathbb C)$&\text{symmetric clans, signature }$(2p,2q)$)\\
$BDI$&$O(p,\mathbb C)\times O(q,\mathbb C)$&\text{orthosymmetric clans, signature }$(p,q)$\\
$DIII$&$GL(n,\mathbb C)$&\text{even skew-symmetric clans}\\
\end{tabular}
\vskip .2in
\begin{tabular}{c c}
\text{\bf type}&\text{\bf open orbit}\\
\hline
$AI$&$[1\ldots n]$\\
$AII$&$[21\ldots 2n,2n-1]$\\
$AIII$&$(1\ldots q+\ldots +q\ldots 1)$\\
$CI$&$(1\ldots n,n\ldots1)$\\
$CII$&$(1\ldots q+\ldots+q\ldots 1)$\\
$BDI$&$(1\ldots q+\ldots+q\ldots 1)$\\
$DIII$&$(1,\ldots,2m,2m-1,2m,\ldots,1,2)$\text{ if }$n=2m$\text{ is even}\\
 & $(1,\ldots,2m,-,+,2m-1,2m,\ldots,1,2)$\text { if }$n=2m+1$\text{ is odd}\\
\end{tabular}
\section{The closure order on symmetric orbits}
\bigskip
The {\sl Bruhat order} on symmetric orbits in a fixed flag variety is defined by containment of closures:  we say that $\mathcal O_1\le\mathcal O_2$ if $\overline{\mathcal O}_1\subseteq\overline{\mathcal O}_2$.  The set of such orbits is called the {\sl Bruhat poset}.  This order was first systematically studied in \cite{RS90}, again building on earlier work  of Springer and others.  It is analogous to the classical Bruhat (or Bruhat-Chevalley) order on the Weyl group $W$, which coincides with the closure order on Schubert varieties.  Recall that this last order makes $W$ into a poset ranked by the length function $\ell(w) = \dim X(w)$.  This poset also has a graph structure and so is also called the {\sl Bruhat graph}.  In this graph the vertices $v,w\in W$ are adjacent if and only if there is a (not necessarily simple) reflection $s$ with $w=sv$ and $\ell(w) > \ell(v)$, or equivalently there is a reflection $t$ with $w=vt$ and $\ell(w) > \ell(v)$. Note that adjacent vertices in the Bruhat graph need {\sl not} have the higher vertex covering the lower one, although all covering relations in this poset correspond to edges in the Bruhat graph.  Denote by $\mathcal O_c,X_c$ the orbit and variety, respectively, corresponding to a vertex $c$ in this graph; in types $AI$ and $AII$ we use the notations $\mathcal O_\pi,X_\pi$ instead of $\mathcal O_c,X_c$, as orbits are parametrized by involutions $\pi$ rather than clans.  We hope there will be no confusion with the notation $X_w$ for the Schubert variety corresponding to the Weyl group element $w$,

Before describing the Bruhat order in the symmetric cases we first recall the Ehresmann-Deodhar-Proctor characterization of Bruhat-Chevalley order on $W$ in the classical cases \cite{E34,D77,P82}.

\begin{theorem}
Given two permutations $v=[v_1\ldots v_n], w=[w_1\ldots w_n]$ we have for the corresponding Schubert varieties $X_v,X_w$ in type $A_{n-1}$ that $X_v\subseteq X_w$ if and only if for all indices $i$ between 1 and $n$, if the terms $v_1,\ldots,v_i$ and $w_1,\ldots,w_i$ are both rearranged in increasing order as $a_1,\ldots,a_i$ and $b_1\ldots b_i$, respectively, then $a_j\le b_j$ for all $j\le i$.
\end{theorem}

\begin{example}
The permutations $[2314]$ and $[4123]$ are incomparable in Bruhat order. Looking at just the first coordinates, we have $2<4$; but then rearranging the first two coordinates as $23,14$, respectively, we find that $2\not<1$.
\end{example}

For signed permutations $[a_1,\ldots,a_n],[b_1,\ldots,b_n]$ in types $B$ and $C$ the criterion is somewhat different.  

\begin{theorem}[\cite{P82,BB05}]
The Schubert varieties $X_v,X_w$ in type $B$ or $C$ corresponding respectively to the signed permutations $v=[v_1,\ldots,v_n],w=[w_1,\ldots,w_n]$ have $X_v\subset X_w$ if and only if for all indices $i\le n$, if we rearrange the terms $v_i,\ldots,v_n$ and $w_i,\ldots,w_n)$ in increasing order as $a_i,\ldots,a_n$ and $b_i,\ldots,b_n$, respectively, then $a_j\ge b_j$ for all $j$.
\end{theorem}

\begin{example}
The signed permutations $v=[-6,-5,-4,-1,2,-3]$ and $w=[-4,6,3,2,5,-1]$ have $v\le w$ in Bruhat order.  Here the respective rearrangements are $[-3]$ and $[-1]; [-3,2]$ and $[-1,5]; [-3,-1,2]$ and $[-1,2,5]; [-4,-3,-1,2]$ and $[-1,2,3,5]; [-5,-4,-3,-1,2]$ and $[-1,2,3,5,6]; [-6,-5,-5,-3,-1,2]$ and $[-4,-1,2,3,5,6]$.
\end{example}

Type $D_n$ is more subtle; the Bruhat order in this case is {\sl not} the restriction of Bruhat order in type $C_n$ to signed permutations with evenly many signs.  

\begin{theorem}[\cite{P82,BB05}]
The Schubert varieties $X_v,X_w$ in type $D$ corresponding to the signed permutations $v=(v(1),\ldots,v(n)),w=(w(1)\ldots,w(n))$ with evenly many minus signs have $X_v\subseteq X_w$ if and only if the condition of the previous theorem on the $a_j,b_j$ holds and, in addition, if the first $k$ terms of $a_i,\ldots,a_n$ and $b_i\ldots,b_n$ have absolute values $1,\ldots,k$ in some order, then the numbers of negative $a_r$ and negative $b_r$ among these first $k$ terms have the same parity.
\end{theorem}

\begin{example}
The signed permutations $(2,3,1)$ and $(3,-2,-1)$ are incomparable in Bruhat order for type $D_3$, since the parity condition is violated.  They correspond to the permutations $(2,3,1,4),(4,1,2,3)$ in type $A_3$ via the standard isomorphism between the root systems of types $A_3$ and $D_3$.  These permutations are incomparable in the Bruhat order for type $A_3$, as we saw above.
\end{example}

The Bruhat poset in the symmetric case is also ranked by (shifted) dimension of orbits.   The closed orbits are the minimal ones and all have the same dimension $d$, equal to the dimension of the flag variety of $K$; if $d$ is subtracted from all orbit dimensions then the resulting function is the rank function \cite[Property 5.12(c),Lemma 7.1]{RS90}.  The Bruhat poset again has the structure of a graph and so is again called the Bruhat graph.  Here however the graph structure is more complicated to describe, being given by the action of certain but not all root reflections on the set of symmetric orbits \cite[\S5]{RS90}. Covering relations in the symmetric poset need {\sl not} correspond to edges in the Bruhat graph.  Rather than describe this graph in general we will do it in each classical case.

\subsection{Type $A$}

In type $AIII$ the Bruhat order on symmetric orbits has been described explicitly by Wyser \cite{W16}.  To begin with, the rank $\ell(c)$ of the orbit $\mathcal O_c$ corresponding to the clan $c=(c_1\ldots c_{p+q})$ is given by
$$
\ell(c)=\sum_{c_i=c_j\in \mathbb N,i<j} (j-i-\#\{k\in\mathbb N: c_s=c_t=k \text{ for some }s<i<t<j\})
$$
\noindent and $\dim\mathcal O_c = (1/2)(p(p-1)+q(q-1))+\ell(c)$ \cite{Y97}. 

\begin{example}
In particular, the closed orbits are exactly the ones whose clans have only signs and all have the same dimension $d_{p,q}=(1/2)(p(p-1)+q(q-1))$.  As mentioned above, the open orbit has clan $(12\ldots q+\ldots+q\ldots1)$, with $p-q$ plus signs, if $p\ge q$, or the same clan with $q-p$ minus signs in the middle, if $q\ge p$.
\end{example} 

Given two clans $c=(c_1\ldots c_{p+q}),d=(d_1\ldots d_{p+q})$ of signature $(p,q)$ for every index $i$ let $c(i;+)$ be the total number of $+$ signs and pairs of equal numbers among $c_1\ldots c_i$ and let $c(i;-)$ be the total number of $-$ signs and pairs of equal numbers among $c_1\ldots c_i$. For all indices $i,j$ with $i<j$ let $c(i;j)$ be the number of pairs of equal numbers $c_s=c_t\in\mathbb N$ with $s\le i<j<t$.  Define $d(i;+),d(i;-),d(i;j)$ similarly.  Then we have \cite[Theorem 1.2]{W16}:

\begin{theorem}
With notation as above, we have $\mathcal O_c\le\mathcal O_d$ if and only if\hfil\linebreak $c(i;+)\ge d(i;+), c(i;-)\ge d(i;-)$, and $c(i;j)\le d(i;j)$ for all indices $i,j$.
\end{theorem}

\begin{example}
Here is the Hasse diagram depicting the Bruhat poset for type $AIII$ in the special case $p=q=2$.  For the sake of clarity we have omitted the edges between vertices whenever the higher one does not cover the lower one and we have not labelled the edges.  Two of the edges, namely those from $1122$ to $1212$ to $1-+1$ and $1+-1$ correspond to edges in the Hasse diagram but {\sl not} in the Bruhat graph; all other edges belong to the Bruhat graph.  There are additional edges in the Bruhat graph not depicted here; e.g. from $+-+-$ to $1-+1$.  As previously observed, edges in the Bruhat graph do not always correspond to covering relations in the Hasse diagram.

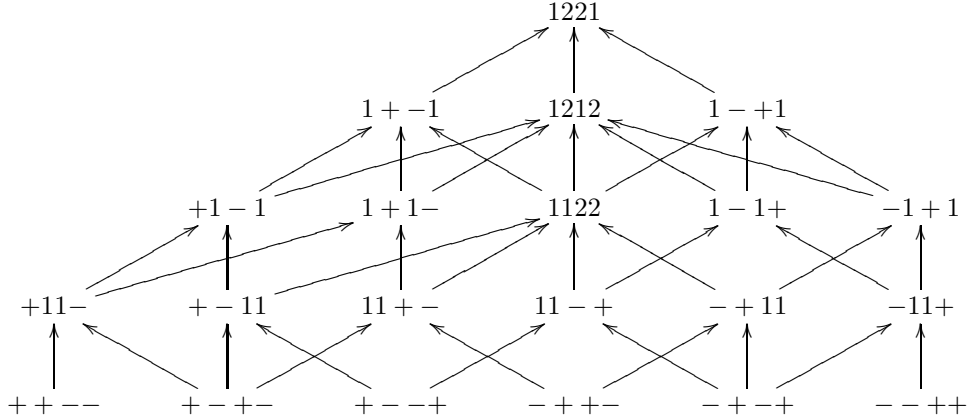
\begin{center}
$\xymatrix{& & & 1221 & & &\\ & & 1+-1\ar[ru]&1212\ar[u]&1-+1\ar[lu]& &\\ & +1-1\ar[ru]\ar[rru]&1+1-\ar[u]\ar[ru]&1122\ar[lu]\ar[u]\ar[ru]&1-1+\ar[lu]\ar[u]&-1+1\ar[llu]\ar[lu] &\\ +11-\ar[ru]\ar[rru]&+-11\ar[u]\ar[rru]&11+-\ar[u] \ar[ru]&11-+\ar[u]\ar[ru]&-+11\ar[lu]\ar[ru]&-11+\ar[lu]\ar[u]\\ ++--\ar[u]&+-+-\ar[lu]\ar[u]\ar[ru]&+--+\ar[lu]\ar[ru]& -++-\ar[lu]\ar[ru]& -+-+\ar[lu]\ar[u]\ar[ru]&--++\ar[u]}
$
\end{center}
\end{example}
\medskip
Wyser also gives a set of combinatorial moves generating the Bruhat order on clans, each replacing a pattern of (not necessarily adjacent) entries in it by another one of the same length:

\begin{enumerate}
\item replace $+-$ by $aa$ (that is, by a pair of equal numbers not equal to any other number appearing)
\item replace $-+$ by $aa$
\item replace $aa+$ by $a+a$ (for $a\in\mathbb N$)
\item replace $aa-$ by $a-a$
\item replace $+aa$ by $a+a$
\item replace $-aa$ by $a-a$
\item replace $aabb$ by $abab$ (for $a,b\in\mathbb N, a\ne b$)
\item replace $aabb$ by $a+-a$
\item replace $aabb$ by $a-+a$
\item replace $abab$ by $abba$
\end{enumerate}
\medskip
\noindent Then we have \cite[Theorem 2.8]{W16}:

\begin{theorem} 
We have $\mathcal O_c< \mathcal O_d$ if and only if $d$ can be obtained from $c$ by a sequence of moves of one of the above types.  
\end{theorem}

Observe that in our example the edges from $1122$ to $1+-1$ and $1-+1$ correspond to the eighth and ninth moves in Wyser's list; the edge from $1122$ to $1221$ corresponds to the tenth move in the list.  In general (in type $AIII$) moves of all types except the last three correspond to edges in the Bruhat graph and all such edges arise in this way.  It is not known whether the order complex of the Bruhat poset is shellable, or whether the poset is $CL$- or $EL$-shellable. 

Turning now to type $AI$, we find that 

\begin{theorem}[{\cite[Example 10.2]{RS90}}]
The closure order on orbits in type $AI$ corresponds to the reverse Bruhat order on $I_n$. 
\end{theorem}

This restricted order is studied in \cite{I04}, where moves analogous to the above moves in type $AIII$ are given generating this order.  The Bruhat poset is ranked by the function
$$
\ell(\pi) = \lfloor n^2/4\rfloor - \sum_{i<\pi(i)} (\pi(i) - i - \#\{k\in\mathbb N: i<k<\pi(i),\pi(k)<i\})
$$  
\noindent for all $\pi\in I_n$ \cite{Y97'}.  An alternative formula \cite[Thm.\ 5.2]{I04}is
$$
\ell(\pi)=\lfloor n^2/4\rfloor -\frac {\text{inv}(\pi)+\text{env}(\pi)} 2
$$
\noindent where inv$(\pi)$ is the number of inversions of $\pi$ and exc$(\pi)$ is the number of excedances of $\pi$ (i.e. the number of indices $i$ such that $\pi(i)>i$). 
Inciitti also shows that this poset is EL-shellable \cite[Thm.\ 6.2]{I04}.

\begin{example}
There is just one closed orbit $\mathcal O_{w_0}$, corresponding to the longest element $w_0$ of $I_n$; it has dimension $m(m-1)$ if $n=2m$ is even and dimension $m^2$ if $n=2m+1$ is odd.  The open orbit corresponds to the identity involution.  
\end{example}

The graph structure on $I_n$ is given by

\begin{proposition}[\cite{M19}]
$\mu,\nu$ in the Bruhat graph are adjacent if and only if either $\nu=t\mu t$ for some transposition $t$ not commuting with $\mu$ or $\nu=t\mu=\mu t$ for some transposition $t$ commuting with $\mu$.
\end{proposition}

\begin{example}
Here is the Hasse diagram for poset in type $AI$ with $n=4$.

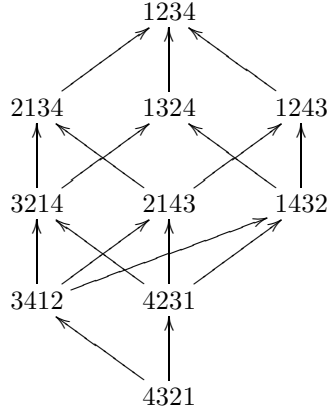
\begin{center}
$\xymatrix{ & 1234 &\\
2134\ar[ru]&1324\ar[u]&1243\ar[lu]\\
3214\ar[u]\ar[ru]&2143\ar[lu]\ar[ru]&1432\ar[lu]\ar[u]\\
3412\ar[u]\ar[ru]\ar[rru]&4231\ar[lu]\ar[u]
\ar[ru]& \\
&4321\ar[lu]
\ar[u]\\}
$
\end{center}
\end{example}

\noindent Very similar results hold for type $AII$.  

\begin{theorem}[\cite{CCT15}]
The Bruhat order on involutions without fixed points in $S_{2m}$ is again the reverse Bruhat-Chevalley order, restricted to $I_{2m}'$.  The rank of the vertex corresponding to an involution $\pi$ is given by the formula of the previous case, replacing $n$ by $2m$.  The vertices $\mu,\nu$ are adjacent in the Bruhat graph if and only if $\nu= t\mu t$ for some transposition $t$ not commuting with $\mu$.   The closed orbit $\mathcal O_{w_0}$ again corresponds to the longest element $w_0$ of $I_{2m}'$ (and of $S_{2m}$) and has dimension $m^2$; the open orbit has involution $(2,1,4,3,\ldots,2m,2m-1)$.
\end{theorem}

\noindent An alternative formula for the rank $\ell(\pi)$ is $n^2-n-\frac {\text{inv}(\pi)-n}  2$\cite[Prop.\ 8]{CCT15}.  The poset is known to be EL-shellable \cite[Thm. 1]{CCT15}.

\subsection{Type $C$}

In type $CII$ the closed orbits are those whose clans $c$ have only signs and each has dimension $p^2+q^2$.  In general, defining $\ell(c)$ a in type $AIII$, the rank $r(c)$ of $\mathcal O_c$ is given by\hfil\break $(1/2)(\ell(c)+\#\{t\in\mathbb N: c_s=c_t\in\mathbb N\text{ and }s\le m<t\le 2m+1-s\})$.  The open orbit has clan\hfil\break $c_{p,q}=(1,2,\ldots,2q+\ldots+2q-1,2q,2q-3,2q-2,\ldots,1,2)$, with $p-q$ plus signs if $p\ge q$, and similarly if instead $q\ge p$.  Wyser's criterion for $\mathcal O_c\le\mathcal O_d$ in type $AIII$ then carries over to this case, except that there is an additional requirement.

\begin{theorem}[\cite{M23}]
Retain the notation of Theorem 3.4.  Given clans $c,d$ parametrizing orbits $\mathcal O_c,\mathcal O_d$ in type $CII$, for every index $i>m$, let $c(i)$ be the number of pairs of equal numbers $c_s,c_t$ with $s<m<t\le 2m+1-s,s\le i$ and define $d(i)$ similarly.  Then $\mathcal O_c\le\mathcal O_d$ if and only if $c(i;+)\ge d(i;+),\hfil\break c(i;-)\ge d(i;-),c(i;j)\le d(i;j)$, and $c(i)\le d(i)$ for all indices $i,j$.
\end{theorem}

The moves generating this order carry over from type $AIII$, except that every move but the seventh must involve a block of at least two entries with indices $i_1,\ldots,i_k$ at most equal to $m$ and must be made simultaneously with its mirror image involving the complementary indices $2m+1-i_k,\ldots,2m+1-i_1$.  Thus an application of the second move sends $(1-+12+-2)$ to $(13312442)$.  The seventh move is allowed to involve a symmetric set of indices (stable under the operation of replacing an index $i$ by $2m+1-i$), in which case it need not be performed simultaneously with any other move.  Thus one application of the seventh move sends  $(+-1122-+)$ to $(+-1212-+)$; another application of this move sends $(11223344)$ to $(12123434)$.  Again, all moves except the last three correspond to edges in the Bruhat graph.

In type $CI$ the rank $r(c)$ is given by the same formula as in type $CII$; the closed orbits are the ones whose clans have only signs, and all have dimension $(1/2)m(m-1)$.  The clan corresponding to the open orbit is $(1,\ldots,m,m,\ldots,1)$.  The criterion for $\mathcal O_c\le\mathcal O_d$ is the same as in the previous case.  In Wyser's list of moves the first, second, and seventh through tenth moves are now allowed to involve a single set of entries symmetric about the midpoint of the clan (and thus to change only the entries in this set); other moves have to be performed simultaneously with their mirror images, as in type $CII$.  Thus one application of the first move sends $(1+-1)$ to $(1221)$; an application of the second move sends $(1-+1)$ to $(1221)$.  We need to add four moves to the list to generate the order, each involving four indices symmetric about the midpoint: 

\begin{itemize}
\item replace $+11-$ by $1212$
\item replace $-11+$ by $1212$
\item replace $+-+-$ by $1212$
\item replace $-+-+$ by $1212$
\end{itemize}

Note that if we have instead say one block of indices $+11-$ {\sl not} symmetric about the midpoint together with its mirror image $+22-$ on the other side of the midpoint (e.g. with the clan $(+11-+22-)$), then we can replace $+11-$ by $1212$ and simultaneously $+22-$ by $3434$ without making use of the added moves.  Indeed, it is enough to replace $+11$ with $1+1$, together with the mirror image of this move, using the fifth move, and then $+-$ by a pair of equal numbers in both blocks, using the first move.  (in our example $(+11-+22-)$ the clan would become $(12123434)$.)  All moves but the last three in Wyser's list and the additional four correspond to edges in the Bruhat graph; all edges arise in this way.

\begin{example}
The Hasse diagram of the poset for type $CI$ with $m=2$ is depicted below.
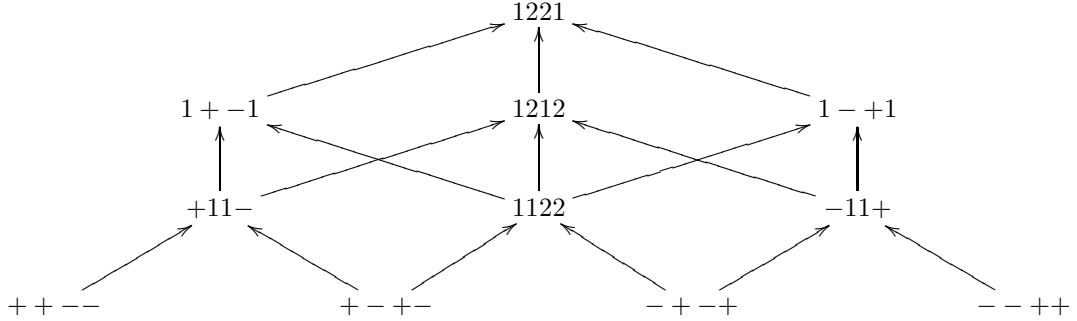
\begin{center}
$\xymatrix{& & &1221 & & &\\& 1+-1\ar[rru]& &1212\ar[u]& &1-+1\ar[llu]& \\&  +11-\ar[u]\ar[rru]& &1122\ar[llu]\ar[u]\ar[rru]& &-11+\ar[u]\ar[llu]& \\ ++--\ar[ru]& &+-+-\ar[lu]\ar[ru]& &-+-+\ar[lu]\ar[ru]& & --++\ar[lu]}$
\end{center}
\end{example}

\subsection{Type $BD$}
\medskip
\noindent In type $BDI$ the rank $r(c)$ of a clan $c=(c_1\ldots,c_n)$ is given as follows.  Let $a(c)$ be the number of indices $s$ with $c_s=c_t\in\mathbb N$ for some $s<{n\over 2}<t<n+1-s$ and $b(c)$ be the number of indices $s$ with $c_s=c_t\in\mathbb N$ for some $s<{n\over2}<t\le n+1-s$.  Then we have
$$
r(c) = \begin{cases}
 (1/2)(\ell(c)-a(c))&\mbox { if }p+q \mbox { is odd}\\
 (1/2)(\ell(c)-a(c))& \mbox{ if }p,q\mbox{ are both even}\\ 
(1/2)(\ell(c) - b(c))&\mbox{ if } p,q\mbox{ are both odd} 
\end{cases}
$$
\noindent and the closed orbits are the ones whose clans have only signs, except that a single pair of equal numbers in the two middle entries is allowed if $p,q$ are both odd.  All closed orbits have dimension equal to that of the flag variety for $K$, as already noted above.  Defining $c(i)$ to be the number of pairs of equal numbers $c_s,c_t$ with $s\le i$ and the index $s$ counted in the term subtracted from $(1/2)\ell(c)$ in the formula for $r(c)$ for an orthosymmetric clan $c$, we have

\begin{theorem}[\cite{M23}]
With notation as above, $\mathcal  O_c\le\mathcal O_d$ if and only if the conditions of Theorem 3.4 hold and $c(i)\ge d(i)$ for all $i\le{n\over2}$.
\end{theorem}

Wyser's moves in type $AIII$ generate the order, except that they cannot involve the middle index in type $B$ \cite{M23}.  Thus for example we are not allowed to move from $(11+22)$ to $(112+2)$, since the latter clan is not orthosymmetric.  Also the four supplementary moves defined above are not allowed in this case, as they depend on starting from a skew-symmetric rather than orthosymmetric clan.   In type $D$, Wyser's eighth and ninth moves are not allowed to affect only the indices in a single pattern symmetric about the midpoint, as that would destroy the symmetry of the clan.  On the other hand, the seventh and tenth of Wyser's moves of indices symmetric about the midpoint. are allowed in this case. Thus  e.g. we are not allowed to move from $(1122)$ to $(1+-1)$, but we are allowed to move from $(1122)$ to $(1212)$..  All moves but the last three in Wyser's list correspond to edges in the Bruhat graph.

Finally, in type $DIII$ the rank $r(c)$ is given by the same formula as in type $BDI$.  The closed orbits are as usual the ones whose clans contain only signs; they all have dimension $(1/2)m(m-1)$.  The clan of the open orbit is $c_o(m,m)=(1,2,\ldots,2k,2k-1,2k,\ldots,1,2)$ if $m=2k$ is even and the same $m$-tuple with $+-$ inserted after the the first $2k$ if $m=2k+1$ is odd.  We denote by $\pm c_o(m,m)$ the clan $c_o(m,m)$ if $m$ is even and either $c_o(m,m)$ or the  (non-even) clan obtained from it by changing all the signs if $m$ is odd.  The criterion for $\mathcal O_c\le\mathcal O_d$ is the same as in type $BDI$.  Wyser's moves, modified as in type $CII$, generate the order, and all moves but the last three correspond to edges in the Bruhat graph \cite{M23}.  Here the last four of Wyser's moves are not allowed to involve a single block of indices symmetric about the midpoint, so that for instance we cannot move from $(1122)$ to $(1212)$, since the former clan is not an even skew-symmetric one.  More generally, an even skew-symmetric clan involving the pattern $1122$ among indices symmetric about the midpoint would be sent by this move to a non-even skew-symmetric clan, so the move is not allowed.  By the same token, the four supplementary moves, involving indices symmetric about the midpoint, are not allowed in this case, as they fail to send even skew-symmetric clans to even skew-symmetric clans.

\section{Criteria for smoothness and rational smoothness}

In this section we begin with a criterion of Kazhdan and Lusztig for rational smoothness of Schubert varieties (for general groups $G$) \cite{KL79,KL80}.  We then give a graph-theoretic criterion for rational smoothness due to Carrell and Peterson \cite{C94}.  Kazhdan and Lusztig's criterion was generalized to the symmetric setting by Lusztig and Vogan \cite{LV80}, while Brion gave a very general version of Carrell-Peterson's criterion for varieties with a torus action (which includes all symmetric varieties)\cite{Br99}. Specializing down to the case of classical groups $G$, we then define pattern avoidance in the symmetric setting and show how it can be used to characterize rational smoothness.  

Recall that an irreducible variety $X$ of dimension $d$ is {\sl rationally smooth} at $x\in X$ the \' etale cohomology with values in the values in the constant $\ell$-adic sheaf $\mathbb Q_\ell$ and support at $x$ is one-dimensional and concentrated in top dimension, so that 
\[
H_x^p(X,\mathbb Q_\ell)=\begin{cases}
0&\mbox{ if }i\ne 2d\\
\mathbb Q_\ell(-d)&\mbox { if }i=2d
\end{cases}
\]
\noindent If $X $is rationally smooth at all of its points we say that it is rationally smooth (without qualification).  If $X$ is a complex projective variety then McCrory has shown that $X$ is rationally smooth if and only if the ordinary cohomology $H^*(X)$ of $X$ over $\mathbb C$ admits Poincar\' e duality, so that its dimension in degree $i$ matches its dimension in degree $2d-i$ for all $i$ \cite{McC77}.  In particular, $X$ is rationally smooth whenever $X$ is smooth.  In \cite{KL79} Kazhdan and Lusztig define a polynomial $P_{x,w}$ in one variable $q$ with nonnegative integer coefficients for every $x,w$ in the Weyl group $W$.  This polynomial is 0 if $x\not\le w$ and 1 if $x=w$; in all other cases it has degree at most $(1/2)(\ell(w)-\ell(x)-1)$, with $\ell(w)$ the usual length function on $W$.  A criterion for Schubert varieties to be rationally smooth at particular points is then given by

\begin{theorem}[{\cite{KL79}}]
$X_w$ is rationally smooth at all points of $C_x$ if and only if the Kazhdan-Lusztig polynomial $P_{x,w}=1$.
\end{theorem}

As the polynomials $P_{x,w}$ are defined by highly recursive formulae, however, we are led to ask for criteria that do not require computing them.  Let $\Gamma$ be the Bruhat graph of $W$.  For every $w\in W$ denote by $\Gamma(w)$ the induced subgraph of the Bruhat poset corresponding to the interval $[1,w]$.  Similarly for $x\le w$ let $\Gamma(x,w)$ be the induced subgraph corresponding to the interval $[x,w]$.  Then we have

\begin{theorem}[{\cite{C94}}]
$X_w$ is rationally smooth (everywhere) if and only if $\Gamma(w)$ is regular, or if and only if all of its vertices have degree $\ell(w)$.
\end{theorem}

\noindent In types $A$ and $C$ (but not in general) it is enough just to compute the degree of one vertex; $X_w$ is rationally smooth if and only if the bottom vertex in $\Gamma(w)$ has degree $\ell(w)$ \cite{P94}.  Actually the graph-theoretic criterion, if imposed on all vertices, holds in general locally.

\begin{theorem}[{\cite{J79,C94}}]
In any type, for $x\le w$, the variety $X_w$ is rationally smooth along $C_x$ if and only if for all $y\in[x,w]$ the number of edges in $\Gamma[x,w]$ joining $y$ to a higher vertex is $\ell(w)-\ell(y)$.
\end{theorem}

\noindent Another criterion refers only to the poset structure.  

\begin{theorem}[{\cite{C94}}]
$X_w$ is rationally smooth if and only if the Poincar\' e polynomial $P_w(t) =\sum_{x\le w} t^{\ell(x)}$ is palindromic.  
\end{theorem}

\begin{remark}
More recently Akyildiz and Carrell have sharpened this result, showing in \cite{AC12} that $P_w(t)$ is a product of polynomials of the form $1+t+\ldots+t^m$ for various $m$ whenever $X_w$ is rationally smooth.
\end{remark}

\noindent There is no analogous poset condition for $X_w$ to be rationally smooth along $C_x$.

\begin{remark}
The condition in \cite[Theorem 6.2.10]{BL00} that the graph $\Gamma(x,w)$ be regular is {\sl not} in fact necessary for $X_w$ to be rationally smooth along $C_x$.  For example, the full flag variety in type $A_3$, corresponding to the long element $[4321]$, is both smooth and rationally smooth at all of its points, including those lying on the orbit indexed by $[2143]$; but the degree of the vertex $[4231]$ in the Bruhat graph $\Gamma_{[2143],[4321]}$ is 5 while the length difference between $4321$ and $2143$ is 4.
\end{remark}

\noindent We also have

\begin{theorem}[{\cite{CK03}}]
For $G$ simply laced (so that all simple roots have the same length), $X_w$ is smooth at all points of $C_x$ if and only if it is rationally smooth at all such points.
\end{theorem}

\noindent This last result fails in the non-simply laced case, in fact already in type $B_2$.

In the symmetric variety setting analogues of the Kazhdan-Lusztig polynomials have been defined by Lusztig and Vogan.  They have shown that a symmetric variety $X=\overline{\mathcal O}$ is rationally smooth along a $K$-orbit $\mathcal O'$ if and only if the polynomial $P_{\mathcal O',\mathcal O}$ attached to the orbits $\mathcal O',\mathcal O$ (and the trivial local system on each one) is the constant function 1, while certain other Lusztig-Vogan polynomials are 0 \cite{LV83,V83}. The graph-theoretic criteria of the previous paragraph have been generalized to a necessary (but not in general sufficient) condition for rational smoothness of the variety $X$.  This was done first by Springer, using the degrees of minimal vertices lying below a given one in the Bruhat graph, if the rank of $G$ equals that of $K$ \cite{S92}.  It was then generalized by Brion to a necessary condition for rational smoothness on the degree of any vertex conjugate under the $W$-action defined in \cite{RS90} to a minimal vertex \cite{S92,Br99}.

\begin{definition}
Let $c$ be a vertex lying below $d$ in the Bruhat graph.  The Bruhat graph $\Gamma(c,d)$ is defined as in the Schubert variety case, using the induced subgroup corresponding to the interval $[c,d]$.  
\end{definition}

\begin{lemma}[{\cite[\S2]{RS90}}]
There is a natural $W$ action on the vertices of the Bruhat graph.
\end{lemma}

\begin{remark}[{\cite[\S10]{RS90}}]
This action is given by conjugation on involutions.  On clans in type $AIII$, coordinate permutations act on clans in the obvious way, by permuting their coordinates.  In the symplectic and orthogonal cases, coordinate permutations act on the entries to the left of the midpoint by permuting these coordinates; they then simultaneously permute the coordinates to the right of the midpoint so as to maintain the symmetry or skew-symmetry of the clan.  In these cases, coordinate sign changes change the corresponding entry in the clan if the corresponding coordinate in the clan is a sign. They then simultaneously change the sign of the corresponding coordinate on the other side of the midpoint.  They act trivially if the corresponding coordinate in the clan is a number.
\end{remark}

\begin{theorem}[{\cite[Theorem 2.5]{Br99}}]
 $X_d$ is rationally smooth at $\mathcal O_c$ only if the degree of $c$ in the graph $\Gamma(c,d)$ equals the rank difference $\ell(d)-\ell(c)$, for all vertices $c$ that are $W$-conjugate to a minimal vertex; in general, this degree is always at least $\ell(d)-\ell(c)$.
 \end{theorem}
 
\noindent We call this condition {\sl Brion's criterion}.  Brion also defines a notion of {\sl attractive slice} $S$ of a symmetric variety $X$ , showing that a necessary and sufficient condition for rational smoothness of a symmetric variety at a point is the smoothness of an attractive slice to it at the corresponding point, and similarly for smoothness \cite[Prop.\ 2.1]{Br99}.  Also Hultman has sharpened Theorem 4.6 in a couple of special cases.  

\begin{theorem}[{\cite{H12}}]
In type $AII$ (and in two exceptional types) $X_\pi$ is rationally smooth at $\mathcal O_\mu$ if and only if the degree of $\mu$ in $\Gamma(\mu,\pi)$ equals the rank difference $\ell(\pi)-\ell(\mu)$; in general, this degree is at least $\ell(\pi)-\ell(\mu)$.  Also $X_\pi$ is rationally smooth if and only if the Poincar\' e polynomial $P_\pi(t)$ is palindromic.
\end{theorem}

\begin{remark} As in the Schubert case the graph $\Gamma(\mu,\pi)$ need not be regular nor have a palindromic Poincar\' e polynomial for $X_\pi$ to be rationally smooth at $\mathcal O_\mu$.
\end{remark}

\subsection{Type $A$}

We now apply these criteria to prove pattern avoidance criteria for smoothness and rational smoothness, treating first symmetric varieties parametrized by clans.  First we need to extend the notion of pattern avoidance to clans.  We say that the clan $c$ {\bf includes} the clan $d=(d_1,\ldots,d_m)$ if there are indices $1\le i_1<\ldots<i_m\le n$ such that  $(c_{i_1},\ldots,c_{i_m})$ is a clan identified with $(d_1,\ldots,d_m)$; otherwise we say that $c$ {\bf avoids} $d$.  Thus for example the clan $(1+2-12-)$ contains $(1-1-)$ and the equivalent clan $(2-2-)$, but avoids $(1221)$.

\begin{theorem}[{\cite{M09}}]
The variety $X_d$ with clan $d$ in type $AIII$ is smooth if and only if it is rationally smooth, or if and only if $d$ avoids the patterns $(1+-1),\hfil\linebreak (1-+1),(1212),(1+221),(1-221),(122+1),(122-1),(122331)$.  Whenever $X_d$ is rationally smooth it is an iterated fiber bundle with smooth fiber over a partial flag variety; whenever $X_d$ is rationally singular this can be detected by Brion's criterion applied to a suitable closed orbit $\mathcal O_c$ below $\mathcal O_d$.
\end{theorem}

\begin{proof}
Suppose first that $c$ contains one of the bad patterns.  If this pattern has just two equal numbers, replace them by $-$ and $+$, in that order; if it includes two such pairs, replace the four numbers by $-,+,-,+$, in that order; if it includes three such pairs, replace the numbers by $-,+,-,+,-,+$, in that order.  In all eight cases, continue by replacing every pair of equal numbers in $c$ by a pair of opposite signs.  We obtain a clan $c$ corresponding to a closed orbit $\mathcal O_c$ below $\mathcal O_d$.  One easily checks that the degree of $\mathcal O_c$ in the Bruhat graph $\Gamma(c,d)$ is larger than $\ell(d)$, whence $X_d$ is rationally singular by Brion's criterion.  

Now suppose that $d$ avoids all the bad patterns.  Writing $d=(d_1\ldots d_n)$ we see that the intervals $[s,t]$ of indices $s,t$ with $d_s=d_t\in\mathbb N$ are such that any two of them are either disjoint or one is contained in the other.  All signs lying between pairs of equal numbers in $d$ are the same.  If a sign lies between a pair of equal numbers, then it also lies between every pair of equal numbers enclosed by the first pair.  Finally, if one pair of equal numbers lies inside another, then the pairs of equal numbers enclosed by this one form a single nested chain.  Then there is a suitable $\iota$-stable parabolic subgroup $Q$ of $G$ containing the Borel subgroup $B$ such that the $K$-orbit $K\cdot\mathfrak q$ of $\frak q$, the Lie algebra of $Q$, identifies with a closed orbit in the partial flag variety $G/Q$, whose preimage $\pi^{-1}(K\cdot\mathfrak q)$ under the natural projection $\pi:G/B\rightarrow G/Q$ is $\mathcal O_d$.   Then $X_d$ fibers smoothly via $\pi$ over $K\cdot\mathfrak q$ with fiber the flag variety $Q/B$ of $Q$, which may be identified with the flag variety of any Levi factor of $Q$ \cite{T05}.  Hence this closure is smooth, as desired.
\end{proof}

\noindent This result was stated incorrectly in \cite{M09}, with some bad patterns missing; the version of that paper on the arXiv is correct.

\begin{example}
The variety corresponding to the clan $d=(1++-2-21)$ in type $AIII$ with $p=q=4$ is rationally singular (not rationally smooth) at points of the orbit corresponding to the clan $c=(+++---+-)$, as one sees by computing the degree of $c$ in the induced Bruhat graph $\Gamma(c,d)$.  In this graph there is an edge from $c$ to any clan $c'$ obtained from $c$ by changing a pair of opposite signs to two equal numbers.  On the other hand, the vertex corresponding to the clan $(12++21)$ in the Bruhat graph of type $AIII$ with $p=4,q=2$ satisfies the hypothesis of the second part of the proof.  The corresponding variety is the full flag variety, so is smooth.
\end{example}

Finally we treat the two cases where $K$-orbits are parametrized by involutions.  Here for typographical convenience we omit the brackets around the one-line notations.  We need to modify the classical definition of pattern inclusion for Schubert varieties.  We decree that an involution $\pi=[\pi_1\ldots\pi_m]$ (still in one-line notation) includes another one $[\mu_1\ldots\mu_r]$ if and only if there are indices $1\le k_1<\ldots<k_r\le m$ {\sl permuted by $\pi$} such that $\pi_{k_a}<\pi_{k_b}$ if and only if $\mu_a<\mu_b$; thus we are interested only in involutions, not arbitrary permutations, lying inside larger involutions.  Thus, for example, the involution $[65872143]$ fails to contain the pattern $[2143]$, since although the indices $2,1,4,3$ occur in that order in the involution they are not permuted by it. We will say more about the distinction between pattern avoidance in the Schubert and symmetric settings below.

\medskip
In type $AII$ we have

\begin{theorem}[{\cite{M11}}]
The variety $X_\pi$ corresponding to the involution $\pi$ is rationally smooth if and only if $\pi$ avoids the bad patterns $351624,64827153,57681324,53281764,\hfil\linebreak43218765,65872143,21654387,21563487,34127856,36154287,21754836,63287154,\hfil\linebreak54821763,46513287,21768435$.  This condition holds if and only if the unique bottom vertex $w_0$ of the Bruhat graph $\Gamma(w_0,\pi)$ has degree $\ell(\pi)$, or if and only if $X_\pi$ is smooth.  More generally, $X_\pi$ is rationally smooth along $\mathcal O_\mu$ if and only if it is smooth along this orbit, or if and only if the bottom vertex of the Bruhat graph $\Gamma(\mu,\pi)$ has degree $\ell(\pi)-\ell(\mu)$.
\end{theorem}

\begin{proof}
We sketch the proof of the first assertion given in \cite{M11}.  If $\pi$ contains one of the bad patterns then one constructs an involution $\mu$ such that the degree of $\mu$ in the Bruhat graph $\Gamma(\mu,\pi)$ is too large, using \cite[Lemma 1]{M11}.  If $\pi$ avoids all bad patterns, then by Theorem 4.7 it is enough to show that the Poincar\' e polynomial $P_\pi(t)$ is palindromic.  In fact one proves something stronger, realizing $P_\pi(t)$ as a product of sums of the form $1+t+\ldots+t^e$ for various exponents $e$.  Let $\pi=\pi_1\ldots\pi_{2n}$ and assume first that $2n-\pi_1\le\pi_{2n-1}$ (i.e., that 1 is closer to the end of $\pi$ than $2n$ is to its beginning).  Set $\pi^{(1)} = t\pi t$, where $t$ is the transposition interchanging $\pi_1$ and $\pi_1+1$, so that $1$ appears one place further to the right in $\pi^{(1)}$ than in $\pi$.  Define $\pi^{(2)},\ldots,\pi^{(2n-\pi_1)}$ similarly, so that 1 appears at the end of $\pi^{(2n-\pi_1)}$.  If $\mu=\mu_1\ldots\mu_{2n}$ then Theorem 3.1 shows that $\mu_1\ge\pi_1$.  If $\mu_1=\pi_1$ then one checks that $\mu'<\pi'$, where $\mu',\pi'$ are obtained from $\mu,\pi$ by omitting the indices $1$ and $\mu_1$, replacing all indices $i$ between 1 and $\mu_1$ by $i-1$, and replacing all indices $j>\mu_1$ by $j-2$; moreover, $\pi'$ continues to avoid all bad patterns.  If instead $\mu_1>\pi_1$, then we claim that $\mu\le\pi^{(1)}$ and that $\pi^{(1)}$ continues to avoid all bad patterns.  If this holds, then induction shows that $\mu\le\pi^{(\mu_1-\pi_1)}$ whence we may as above eliminate the indices $1$ and $\mu_1$ from $\mu$ and $\pi^{(\mu_1-\pi_1)}$ and repeat the above procedure.  Using the first formula for the rank function in $I_{2m}$ given after the statement of Theorem 3.6, we deduce that $P_\pi(t)$ factors in the way claimed above, where the first factor is $1+t+\ldots+t^{2m-\pi_1}$.

To prove the claim that $u\le\pi^{(1)}$ and that $\pi^{(1)}$ avoids the bad patterns, set $\pi_1=k,\pi_{k+1} = i$ and suppose that there is $\mu$ with $\mu<\pi,\mu\not\le\pi^{(1)}$, and $\mu_1>\pi_1$.  There are two cases, depending on whether $i<k$ or $i>k+1$.  If $i<k$ then we look at the indices greater than $k$ among $\pi_1\ldots\pi_k$.  If these do not occur in increasing order, then the pattern $p:=465132$ is included in $\pi$, in such a way that the 4 corresponds to $\pi_1=k$.  The assumption $2n-\pi_1\le\pi_{2n}-1$ implies that $\pi\ne p$, so that $\pi$ is the product of three disjoint transpositions forming the pattern $p$ and at least one more transposition.  Now one checks that no matter how one chooses this transposition to guarantee that $2n-\pi_1\le\pi_{2n}-1$ we get a bad pattern in $\pi$, a contradiction; more precisely, one of the five patterns $46513287,63287154, 65872143, 64827153$, or $57681324$, must occur in $\pi$.  If the indices greater than $k$ do occur in increasing order, then (since $\pi$ is an involution) the indices less than $k$ {\sl not} occurring among $\pi_1\ldots\pi_k$ are all larger than $i$, whence the indices $2,\ldots, i-1$ occur among $\pi_1\ldots \pi_{k-1}$ (and $\pi_k=1$).  These conditions are incompatible with $\mu<\pi$ and $\mu\not\le\pi^{(1)}$, so this case leads to a contradiction.  So we must have $i>k+1$.  Now if $\pi_j>k$ for any $j<i$ then one of the patterns $351624$ of $365412$ must occur in $\pi$; the former is ruled out since it is a bad pattern and the latter, combined with the condition that $2n-\pi_1\le\pi_{2n}-1$, would force one of the bad patterns $36154287$ or $53281764$ to occur in $\pi$ (arguing as in the case above where $\pi$ includes the pattern $p$).  So $\pi_1\ldots\pi_k$ must be a permutation of $1\ldots k$ and $k$ is even.  Now the absence of the patterns $43218765$ and $43217856$ in $\pi$ implies that $i=k+2$.  In this case the only way that we can have $\mu<\pi,\mu\not\le\pi^{(1)}$ is if the $k-1$ indices between 2 and $k$ appear among $\mu_2\ldots\mu_k$, which is a contradiction since $\mu$ has no fixed points.

If instead $\pi_{2n}-1<2n-\pi_1$, then one repeats the above argument, replacing $1$ by $2n$ and moving $2n$ to the left instead of 1 to the right.  Thus we define $\pi^{(1)},\pi^{(2)}$, and so on, so that $2n$ appears one place to the left in $\pi^{(1)}$ than it does in $\pi$; if $\mu\le\pi$ then we must have $\mu_{2n}\le\pi_{2n}$, and if $\mu_{2n}<\pi_{2n}$, then we must have $\mu\le\pi^{(1)}$, lest $\pi$ contain a bad pattern.  Here the two \lq\lq bad seeds'' that must be ruled out are $546213$ and $532614$; these give rise to the bad patterns $21768435, 54827163,\linebreak65872143,64827153,57681324, 21754386$, and $53281764$.

Finally, we must ensure in both cases that $\pi^{(1)}$ avoids all bad patterns whenever $\pi$ does.  This requires that we rule out four more \lq\lq bad seeds", namely $216543,432165,\linebreak215634$, and $341265$; we achieve this by ruling out the bad patterns $21654387,\linebreak43218765,34127856,43217856,21563487$, and $34128765$.  Excluding also the bad pattern $351624$ of length 6, we see that if $\pi$ avoids all bad patterns then $P_\pi(t)$ factors in the desired way and $X_\pi$ is rationally smooth, as required.
\end{proof}
We refer to \cite{M11} for the proof of the second and third assertions. 
 \medskip
In type $AI$ we have

\begin{theorem}
The variety $X_\pi$ corresponding to $\pi$ is rationally smooth if and only if $\pi$ avoids the bad patterns $14325,21543,32154,154326,124356,351624,132546,\hfil\linebreak426153,153624,351426,1243576,2135467,2137654,4321576,5276143,5472163,\hfil\linebreak1657324,4651327,57681324,65872143,13247856,34125768,34127856,64827153$ and $2143$, except that the pattern $2143$ is allowed to occur if there are an odd number of fixed indices between $21$ and $43$ (thus the involution $21354$ corresponds to a rationally smooth variety while $213465$ does not).  The variety $X_\pi$ is smooth if and only if it is rationally smooth and in addition avoids the patterns $1324,2143$ (regardless of the number of fixed points between $21$ and $43$ for the latter pattern).  Rational singularity is always detected by the degree of some vertex conjugate to the bottom vertex in the Bruhat graph.
\end{theorem}

The idea of the proof is as follows.  First, if a bad pattern is included, then one shows directly that some vertex in the Bruhat graph has too large a degree; this is done inductively, starting with the graph attached to the bad pattern itself and then showing that such a vertex continues to exist in the graph as fixed points and flipped pairs of indices are added to this pattern.  Next one shows that if one just assumes that all vertices conjugate to the bottom one in the Bruhat graph have the right degree, then $X_\pi$ is rationally smooth; to do this one uses the notion of a slice from \cite[2.1]{Br99} to construct a variety whose rational smoothness or singularity along the unique minimal closed suborbit matches that of $X(\pi)$ and then checks directly that the slice is indeed rationally smooth along this suborbit.  Finally one shows that avoidance of the bad patterns is equivalent to all vertices conjugate to the bottom vertex having the right degree.  One does this by showing that if the degree of some vertex conjugate to the bottom one is too large, then this must already be the case for the graph attached to some subinvolution contained in $\pi$ corresponding to a $K$-orbit in the flag variety of $GL(8,\mathbb C)$.  One then appeals to a computer calculation to show that rational singularity of $K$-orbits in this flag variety is indeed captured by the above list of patterns.  The smoothness criterion also follows by looking at slices.  For the details see \cite{M19,M20}.  If $n$ is even, then it suffices to look at the degree of the bottom vertex $w_0$ alone to detect rational singularity.  If $n$ is odd, it suffices to look at the degrees of just $n$ vertices conjugate to the bottom vertex (including the vertex itself) to detect rational singularity.
  
\begin{conjecture}
In types $AI$ and $AII$, an involution $\pi$ avoids all bad patterns in the sense of this section if and only if it avoids all bad patterns in the sense used for Schubert varieties (where it is not required that the indices in the pattern be permuted by the involution).
\end{conjecture}

\begin{example}
As noted above, the involution $\pi=[65872143]$ fails to contains the pattern $[2143]$ in the Schubert variety sense, but it is one of the bad patterns of Theorem 4.13, so it corresponds to a rationally singular symmetric variety in type $AII$.  Substantial progress on this conjecture has recently been made by Fang, Hamaker, and Troyka \cite{FHT20}.
\end{example}

\subsection{Type $C$}

In type $CII$ the pattern avoidance condition is imposed on just part of the clan.  We have

\begin{theorem}[{\cite{MT09}}]
The variety $X_d$ is smooth if and only if it is rationally smooth.  This occurs if and only if the clan $d$ takes the form $(\gamma_1,d_{p',q'},\gamma_1')$, the concatenation of the clans $\gamma_1,d_{p',q'}$, and $\gamma_1'$, where these clans are specified as follows.  First, we have $0\le p'\le p,0\le q'\le q$ and $d_{p',q'}$ is the clan of the open orbit for the case $K=Sp(2p',\mathbb C)\times Sp(2q',\mathbb C)$ (taken to be the empty clan if $p'=q'=0$).   Next, $\gamma_1$ is a clan for type $AIII$ of signature $(p-p',q-q')$ avoiding the bad patterns of Theorem 4.8.  Finally, $\gamma_1'$ is the unique clan making $(\gamma_1,d_{p',q'},\gamma_1')$ symmetric.  Whenever $X_d$ is rationally smooth it is an iterated fiber bundle with smooth fiber over a partial flag variety, so that $X_d$ is smooth; whenever $X_d$ is rationally singular this can be detected by Brion's criterion applied to a suitable closed orbit below $d$.
\end{theorem}

\noindent The proof is similar to that of the previous result (see also \cite{M20'}).  Again, the paper \cite{MT09} states the theorem incorrectly, with some bad patterns missing; the version on the arXiv is correct.

It is easy to construct examples of involutions containing one of the bad patterns for which one can check explicitly that some vertex in the relevant Bruhat graph has too large a degree for rational smoothness, along the lines of the proof of Theorem 4.8.  The statement of this theorem is slightly more complicated than that of Theorem 4.8 since the bad patterns occurring in the statement of that theorem occur in clans corresponding to the full flag variety, which is clearly smooth.

In type $CI$ we have a similar criterion which is a bit more difficult to state.  Here, for the first time in the symmetric setting, smoothness and rational smoothness diverge.

\begin{theorem}[{\cite{M20'}}]
The variety $X_d$ is rationally smooth if and only if the clan $d$ takes the form $(\gamma_1,\gamma_0,\gamma_1'$), where $\gamma_1,\gamma_0,\gamma_1'$ are specified as follows.  First, $\gamma_0$ is either empty, $(1+-1),(1-+1)$, or $(1212)$. Next, $\gamma_1$ is a clan for type $AIII$ avoiding the bad patterns of Theorem 4.8, possibly followed by a string of distinct integers appearing only once in it if $\gamma_0$ is empty.  Finally, $\gamma_1'$ is chosen to make $d$ skew-symmetric.  $X_d$ is smooth if and only if $d$ takes the above form with $\gamma_0$ either empty or $1212$.  If $X_d$ is rationally singular this can be detected by Brion's criterion.
\end{theorem}

\begin{proof}
If $d$ does not take the given form then one argues as above that the degree of a suitable closed orbit $\mathcal O_c$ in the Bruhat graph $\Gamma(c,d)$ is too large.  If it does take this form then as above $X_d$ is seen to be a fiber bundle with smooth fiber over $X_{\gamma_0}$.  By the Leray-Hirsch Theorem \cite[p. 258]{S66} it suffices to verify the conclusion for $\gamma$ equal to any of the three nonempty possibilities for $\gamma_0$.  So let $G=\,$Sp$(4,\mathbb C)$ and fix a basis $\{e_1,e_2\}$ for a maximal isotropic subspace $S$ of $\mathbb C^4$, equipped with a nondegenerate skew-symmetric bilinear form $(\cdot,\cdot)$.  Let $f_1,f_2$ be the dual basis for an isotropic dual to $S$.  Then $G/B$ may be identified with the space of maximal isotropic flags $0\subset V_1\subset V_2$ in $\mathbb C^4$.  First let $\gamma =(1+-1)$.  The set of maximal isotropic flags for which $V_1$ is spanned by $e_1+xf_1+yf_2$ and $V_2$ is spanned by $V_1$ and $ze_2+yze_1+wf_2$ for some complex affine coordinates $x,y$ and projective coordinates $z,w$ is a slice in the sense of \cite[Definition 2.1]{Br99} to $X_{\gamma}$ at the point corresponding to $x=y=w=0,z=1$, where the coordinates lie on the variety with equation $xw-y^2z=0$.  A similar argument applies to the case $\gamma=(1-+1)$.  Finally, if $\gamma=(1212)$ then $X_{\gamma_0}$ admits such a slice with the equation $xw-y^2z=0$ replaced by $x=0$.  The result then follows by inspection.
\end{proof}

\subsection{Type $BD$} 

The last two types with the $K$-orbits parametrized by clans are $BDI$ and $DIII$.  In type $BDI$ we have

\begin{theorem}
The variety $X_d$ is rationally smooth if and only if the clan $d$ takes the form $(\gamma_1,\gamma_0,\gamma_1')$, and $\gamma_1,\gamma_0$, and $\gamma_1'$ are specified as follows.  First, $\gamma_0$ is either empty, $(1+-+1),(1-+-1)$, or $(12\pm12)$.  Next,  $\gamma_1$ is a clan for type $AIII$ avoiding the bad patterns of Theorem 4.8, possibly followed by a string of distinct numbers appearing only once in $\gamma_1$ and then a string of equal signs if $\gamma_0$ is empty.  Finally, $\gamma_1'$ is chosen to make $c$ orthosymmetric.  $X_d$ is smooth if and only if $\gamma_0$ is either empty, a string of equal signs, or $(12\pm12)$.  Once again Brion's criterion detects rational singularity in all cases.
\end{theorem}

\noindent The proof is similar to that of the previous result; see \cite{M23}.
\medskip
In type $DIII$ we have

\begin{theorem}[{\cite{MT09}}]
The variety $X_d$ is smooth if and only if it is rationally smooth, or if and only if the clan $d$ takes either the form $(\gamma_1,\pm c_o(m',m'),\gamma_1')$ or $(1,\gamma_1,2,1\gamma_1',2)$, where  in the first case $\gamma_1$ is a clan (in type $AIII$ as usual) avoiding the bad patterns of Theorem 4.1 and the sign of $c_o(m',m')$ is chosen to make $d$ even; in the second case $\gamma_1$ avoids the bad patterns of Theorem 4.1 {\sl and} the patterns $(\pm33),(33\pm)$, and $(2233)$, and in both cases $\gamma_1'$ is chosen to make $d$ skew-symmetric.  Brion's criterion detects rational singularity in all cases.
\end{theorem}

\section{Richardson varieties and clans in type $AIII$}
\bigskip
Since the appearance of \cite{M09}, Wyser and Woo have further studied symmetric varieties in type $AIII$ in \cite{WW15}, studying the singular locus of a singular variety and also geometric properties other than rational singularity.  We conclude this chapter with a brief account of their work.  

We begin with some general remarks about Richardson varieties.  Fix a complex semisimple group $G$ with Borel subgroup $B$ and Weyl group $W$.  Choose a Borel subgroup $B^-$ opposite to $B$ in $G$ (so that the intersection $B\cap B'$ is a maximal torus in $G$).   Elements $w\in W$ then parametrize Schubert cells $C_w$, Schubert varieties $X_w$, opposite Schubert cells $C_w^-$ (that is, $B^-$-orbits in $G/B$), and opposite Schubert varieties $X_w^-$.  The intersection $X_u\cap X_v^-$ of a Schubert variety and an opposite Schubert variety is called a {\sl Richardson variety}.  Now we briefly consider properties more general than smoothness and rational smoothness.   A property $\mathcal P$ of the points in a variety $V$ is called {\sl local, open}, and {\sl multiplicative} if it depends only on the local ring at a point, holds on a nonempty open subset of $V$, and holds at every points of a product $X\times Y$ if and only if it holds at every point of $X$ and $Y$.  Then we have

\begin{theorem} \cite{KWY13}
An open local multiplicative property $\mathcal P$ holds at all points of a Richardson variety $X_u\cap X^-_v$ with $v\le u$ if and only if it holds on $X_u$ along $C_v$ and holds on $X^-_v$ along $C^-_u$, or equivalently holds on $X^-_v$ along $C'_u$ and holds on $X^-_{w_0u}$ along $C^-_{w_0v}$.
\end{theorem}

\noindent and

\begin{theorem} \cite{KWY13}
If $\Sigma(X)$ denotes the singular locus of a variety $X$, then we have $\Sigma(X_u\cap X_v^-) = (\Sigma(X_v^-)\cap X_u)\cup(\Sigma(X_u)\cap X_v^-)$.  In particular, the singular locus of a Richardson variety is a union of Richardson varieties.
\end{theorem}

The connection between symmetric varieties in type $AIII$ and Richardson varieties arises from the following result.  To state it we need some notation.  Given a clan $c$ avoiding $1212$, we attach to it a pair $u(c),v(c)$ of permutations as follows.  The one-line notation of $v(c)$ is obtained from $c$ by first listing in ascending order the positions of $c$ containing a $+$ or the first occurrence of a number and then listing in ascending order the position with a $-$ or the second occurrence of a number.  Similarly, $u(c)$ is obtained by first listing in ascending order the position of $c$ with a $+$ or the second occurrence of a number, followed by listing in ascending order the positions with a $-$ or the first occurrence of a number.  For example, if $c=(12+-12)$ then $v(c)=123456,u(c)=356124$.  If $c$ has signature $(p,q)$ then denote by $w_0^K$ the long element of $W_K=S_p\times S_q$, reversing the indices $1,\ldots,p$ and $p+1,\ldots,p+q$.  Then we have

\begin{theorem} \cite{W13} 
With notation as above the variety $X(c)$ coincides with the Richardson variety $X_u\cap X'_v$, where $u=w_0^Ku(c)^{-1},v=v(c)^{-1}$.  Moreover the permutations $u,v$ are Grassmannian (that is, they have just one possible descent, in position $p$).
\end{theorem}

Now it is well known that Kazhdan-Lusztig polynomials of Grassmannian permutations can be computed by Lascoux-Sch\" utzenberger path diagrams  (see \cite{BL00,WW15}).  Thus one can use path diagrams to determine which of the $1212$-avoiding clans $d$ correspond to rationally smooth varieties $X_d$.  This is done by Woo and Wyser in \cite{WW15}, thereby recovering the criterion of \cite{M09}, bearing in mind that clans containing $1212$ correspond to rationally singular varieties.  Woo and Wyser  then go on to compute the singular locus of $X_d$ whenever $d$ avoids $1212$, using path diagrams.   They then consider another measure of singularity, namely lci-ness.  Recall that a variety or scheme $X$ is said to be {\sl lci at the point $P$} if the local ring of $X$ at $p$ is a local complete intersection, so that the maximal ideal at this point is generated by a regular sequence.  Using an unpublished criterion of Darayon for the Schubert variety of a Grassmannian permutation to be lci at every point based on its path diagram, Woo and Wyser produce the following list of bad patterns for lci-ness.

\begin{theorem} \cite{WW15}
If $d$ avoids $1212$, then $X_d$ is lci at all points if and only if $d$ avoids the patterns in the following list, together with their negatives (obtained by changing all the signs).
$(1++-1),(1+--1),(1-22+1),(1++221),(1+-221),\hfil\linebreak
(122--1),(122+-1),(12+2-1),(1+2-21),(1+23321),(12332-1),(1-22331),\hfil\linebreak(12+2331),(122+331),(1223-31),
(12233+1),(1223-31),(12233+1),(12332441),\hfil\linebreak(12234431),(12233441)$

\end{theorem}

The non-lci locus of a $1212$-avoiding variety, like its rationally singular locus, can be computed from its path diagram.  Using Macaulay 2, Woo and Wyser have studied which $1212$-containing varieties fail to be lci at some point.  They found that non-lci-ness is characterized by pattern avoidance for $p+q\le8$, for both $1212$-avoiding and $1212$-containing varieties alike, but were unable to push the computations far enough to fully cover even the case $p=7,q=2$, though they did conjecture that non-lci-ness in general can be characterized by pattern avoidance.  They also studied Gorensteinness of varieties, but found that this did not seem to be characterized by pattern avoidance, in either the $1212$-avoiding or $1212$-containing cases, though it can be read off from the path diagram; the lack of a pattern avoidance criterion is for this property is known for Schubert varieties \cite{WY08}.  


\begin{thebibliography}{FHT20}
 
 \bibitem[AB14]{AB14} H.\ Abe and S.\ Billey, \textsl{Consequences of the Lakshmibai-Sandhya Theorem:  the ubiquity of permutation patterns in Schubert calculus and related geometry}, arXiv:1403.4345, to appear in Proc.\ Math.\ Soc.\ of Japan.
 
 \bibitem[AC12]{AC12} E.\ Akyildiz and J.\ Carrell, \textsl{Betti numbers of smooth Schubert varieties and the remarkable formula of Kostant, Macdonald, Shapiro, and Steinberg}, Mich.\ Math.\ J.~{\bf61} (2012), 543--553.
 
 \bibitem[BB81]{BB81} A.\ Beilinson and J.\ Bernstein, \textsl{Localization de $\frak g$-modules}, C.\ R.\ Acad.\ Sci.\ Paris {\bf292} (1981), 15--18.
 
 \bibitem[B98]{B98} S.\ Billey, \textsl{Pattern avoidance and rational smoothness for Schubert varieties}, Adv.\ Math.~{\bf139} (1998), 141--156.
 
 
 
 
 \bibitem[BL00]{BL00} S.\ Billey and V.\ Lakshmibai, \textsl{Singular loci of Schubert varieties}, Progress in Math.\ {\bf182}, Birkh\"auser, Boston, 2000.
 
 \bibitem[BB05]{BB05} A.\ Bj\" orner and F.\ Brenti, \textsl{Combinatorics of Coxeter groups}, Graduate Texts in Mathematics {\bf231}, Springer, New York, 2005.

\bibitem[BoB85]{BoB85} W.\ Borho and J.-L.\ Brylinski, \textsl{Differential operators on homogeneous spaces III: Characteristic varieties of Harish-Chandra modules and of primitive ideals}, Inv.\ Math.~{\bf80} (1985), 1-68.


\bibitem[Br99]{Br99} M.\ Brion, \textsl{Rational smoothness and fixed points of torus actions}, Transf.\ Groups {\bf4} (1999), 127--156.

\bibitem[CCT15]{CCT15} M.B.\ Can, Y.\ Chernizvsky, and T.\ Twelbeck, \textsl{Lexicographic shellability of the Bruaht-Chevalley order on fixed-point-free involutions}, Isr.\ J.\ Math.~{\bf207} (2015), 281--299.

\bibitem[C94]{C94} J.\ Carrell, \textsl{The Bruhat graph of a Coxeter group, a conjecture of Deodhar, and rational smoothness of Schubert varieties}, Proc.\ Symp.\ Pure Math.~{\bf56} (1994), 53--61.

\bibitem[CK03]{CK03} J.\ Carrell and J.\ Kuttler, \textsl{On the smooth points of $T$-stable varieties in $G/B$ and the Peterson map}, Inv.\ Math.~{\bf151} (2003), 353--379.

\bibitem[D77]{D77} V.\ Deodhar, \textsl{Some characterizations of Bruhat ordering on a Coxeter group and determination of the relative M\" obius function}, Inv. Math.~{\bf39} (1977), 187--198. 

\bibitem[E34]{E34} C.\ Ehresmann, \textsl{Sur la topologie de certains espaces homog\' enes}, Ann. of Math.~{\bf35} (1934), 396--443.

\bibitem[FHT20]{FHT20} J.\ Fang, Z.\ Hamaker, and J.\ Troyka, \textsl{On pattern avoidance in matchings and involutions}, arxiv:2009.00079v1.

\bibitem[He78]{He78} S.\ Helgason, \text{Differential geometry, Lie groups, and symmetric spaces}, Academic Press, New York, 1978.

\bibitem[H12]{H12} A.\ Hultman, \textsl{Criteria for rational smoothness of some symmetric orbit closures}, Adv.\ Math.~{\bf229} (2012), 183--200.

\bibitem[I04]{I04} F.\ Incitti, \textsl{The Bruhat order on the involutions of the symmetric group}, J.\ Alg.\ Comb.~{\bf20} (2004), 243--261.


\bibitem[J79]{J79} J.\ C.\ Jantzen, {\textsl Moduln mit einem h\" ochsten Gewicht}, Lecture Notes in Math.~{\bf750}, Springer, Berlin, 1979.

\bibitem[KL79]{KL79} D.\ Kazhdan and G.\ Lusztig, \textsl{Representations of Coxeter groups and Hecke algebras}, Inv.\ Math.~{\bf53} (1979), 165--184.

\bibitem[KL80]{KL80} D.\ Kazhdan and G.\ Lusztig, \textsl{Schubert varieties and Poincar\' e duality}, in Geometry of the Laplace operator, Proc.\ Symp.\ Pure Math.~{\bf36} (1980), American Mathematical Society, Providence, 185--203.

\bibitem[KWY13]{KWY13} A.\ Knutson, A.\ Woo, and A.\ Yong, \textsl{Singularities of Richardson varieties}, Math.\ Res.\ Lett.~{bf20} (2013), 391--400.



\bibitem[LS90]{LS90} V.\ Lakshmibai and B.\ Sandhya, \textsl{Criterion for smoothness of Schubert varieties in $SL(n)/B$}, Proc.\ Indian Acad.\ Sci.~{\bf100} (1990), 45--52.


\bibitem[LV83]{LV83} G.\ Lusztig and D.\ A.\ Vogan, \textsl{Singularities of closures of $K$-orbits on flag manifolds}, Inv.\ Math.~{\bf71} (1983), 365--379.


\bibitem[MO88]{MO88} T.\ Matsuki and T.\ Oshima, \textsl{Embedding of discrete series into principal series}, in The Orbit Method in Representation Theory, Progress in Math.\ {\bf82}, Birkh\"auser, Boston, 1988, 147--175.

\bibitem[McC77]{McC77} C.\ McCrory, \textsl{A characterization of homology manifolds}, J.\ London Math.\ Soc.~{\bf16} (1977), 146--159.

\bibitem[M09]{M09} W.\ McGovern, \textsl{Closures of $K$-orbits in the flag variety for $U(p,q)$}, J.\ Alg.\ {\bf322} (2009), 2709--2712; arXiv:0905.0127.

\bibitem[M11]{M11} W.\ McGovern, \textsl{Closures of $K$-orbits in the flag variety for $SU^*(2n)$},  Rep.\ Theory.\ {\bf15} (2011), 568--573. 


\bibitem[M19]{M19} W.\ McGovern, \textsl{Closures of $O_n$-orbits in the flag variety for $GL_n$}, in Representations and Nilpotent Orbits of Lie Algebraic Systems:  in honour of the 75th birthday of Tony Joseph, Prog.\ in Math.~{\bf330} (2019), Birkh\" auser, Boston, 411--419.

\bibitem[M20]{M20} W.\ McGovern, \textsl{Closures of $O_n$ orbits in the flag variety for $GL_n$, II}, arXiv:2010.07114.

\bibitem[M20']{M20'} W.\ McGovern, \textsl{Closures of $K$-orbits in the flag variety for $Sp(2n,\mathbb R)$}, in Lie Theory and its Applications to Physics, Varna, Bulgaria, Proceedings in Mathematics and Statistics {\bf335} (2020), Springer, New York, 359--364.

\bibitem[M23]{M23} W.\ McGovern, \textsl{Representation Theory and Geometry of the Flag Variety}, Studies in Mathematics {\bf90}, De Gruyter, Berlin, 2023.

\bibitem[MT09]{MT09} W.\ McGovern and P.\ Trapa, \textsl{Pattern avoidance and smoothness of closures for orbits of a symmetric subgroup in the flag variety}, J.\ Alg.\ {\bf322} (2009), 2713--2730; arXiv:0904.4493.

\bibitem[P94]{P94} P.\ Polo, \textsl{On Zariski tangent spaces of Schubert varieties, and a proof of a conjecture of Deodhar}, Indag.\ Math.~{\bf5} (1994), 483--493.

\bibitem[P82]{P82} R.\ Proctor, \textsl{Classical Bruhat orders are lexicographic shellable}, J.\ Alg.\ {\bf77} (1982), 104--126. 

\bibitem[RS90]{RS90} R.\ W.\ Richardson and T.\ A.\ Springer, \textsl{On the Bruhat order for symmetric varieties}, Geom.\ Dedicata {\bf35} (1990), 389--436.

\bibitem[S66]{S66} E.\ H.\ Spanier, \textsl{Algebraic Topology}, Springer, New York, 1966.

\bibitem[S92]{S92} T.\  A.\ Springer, \textsl{ A combinatorial result on $K$-orbits in a flag manifold}, in Proceedings, Sophus Lie Memorial Conference, Oslo, 1992 Scandinavian University Press, 1992, 363--370.

\bibitem[T05]{T05} P.\ E.\ Trapa, \textsl{Richardson orbits for real classical groups}, J.\ Alg.~{\bf286} (2005), 361--385.

\bibitem[V83]{V83} D.\ Vogan, \textsl{Irreducible characters of semisimple Lie groups III: proof of Kazhdan-Lusztig conjecture in the integral case}. Inv.\ Math.\ {\bf71} (1983), 381--417.



\bibitem[W89]{W89} J.\ Wolper, \textsl{A combinatorial approach to the singularities of Schubert varieties}, Adv.\ in Math.~{\bf76} (1989), 184--193.

\bibitem[WW15]{WW15} A.\ Woo and B.\ Wyser, \textsl{Combinatorial results on $(1,2,1,2)$-avoiding $GL(p,\mathbb C)\times GL(q,\mathbb C)$-orbit closures on $GL(p+q,\mathbb C)/B$}, Int.\ Math.\ Research Notices~{\bf24} (2015), 13148-13193.


\bibitem[WY08]{WY08} A.\ Woo and A.\ Yong, \textsl{Governing singularities of Schubert varieties}, J.\ Alg.~{\bf320} (2008), 495--520.

\bibitem[W13]{W13} B.\ Wyser, \textsl{Schubert calculus of Richardson varieties stable under spherical Levi subgroups}, J.\ Alg.\ Comb.~{\bf38} (2013), 829--859.

\bibitem[W16]{W16} B.\ Wyser, \textsl{The Bruhat order on clans}, J.\ Alg.\ Comb.~{\bf44} (2016), 495--517.


\bibitem[Y97]{Y97} A.\ Yamamoto, \textsl{Orbits in the flag variety and images of the moment map for classical groups I}, Rep.\ Theory {\bf1} (1997), 327--404.

\bibitem[Y97']{Y97'} A.\ Yamamoto, \textsl{Orbits in the flag variety and images of the moment map for classical groups II}, preprint, 1997.

\end{thebibliography}
\end{document}